\documentclass[12pt]{article}
\usepackage{amsmath, amsthm,amssymb,latexsym,enumerate}
\usepackage{graphicx}
\usepackage{hyperref}
\usepackage{xcolor}
\usepackage{cite}
\usepackage{dsfont}
\usepackage[capitalise]{cleveref}

\usepackage[normalem]{ulem}
 
\DeclareMathOperator{\diff}{diff}
\DeclareMathOperator{\mad}{mad}

\DeclareMathOperator{\EEu}{EE}
\DeclareMathOperator{\OEu}{OE}

\newtheorem{theorem}{Theorem}[section]
\newtheorem*{AlonTarsi}{Alon-Tarsi Theorem}
\newtheorem{lemma}[theorem]{Lemma}
\newtheorem{corollary}[theorem]{Corollary}
\newtheorem{claim}[theorem]{Claim}

\newtheorem{observation}[theorem]{Observation}

\newtheorem{definition}[theorem]{Definition}
\numberwithin{equation}{section}

\usepackage[colorinlistoftodos,prependcaption,textsize=scriptsize, textwidth=20mm,]{todonotes}

\title{On Alon-Tarsi orientations of sparse graphs}

\author{Eun-Kyung Cho\thanks{
Department of Mathematics, Hanyang University, Seoul, Republic of Korea.
 \texttt{ekcho2020@gmail.com}
}
\and Ilkyoo Choi\thanks{
Department of Mathematics, Hankuk University of Foreign Studies, Yongin-si, Gyeonggi-do, Republic of Korea. 
 and  Discrete Mathematics Group, Institute for Basic Science (IBS), Daejeon, Republic of Korea.
 \texttt {ilkyoo@hufs.ac.kr}
 }
\and Boram Park\thanks{Department of Mathematics Education, Seoul National University,  Gwanak-Ro 1, Gwanak-Gu, Seoul,  Republic of Korea. \texttt{borampark22@gmail.com}
}
\and Xuding Zhu\thanks{
School of Mathematical Sciences, Zhejiang Normal University, China.
\texttt{xdzhu@zjnu.edu.cn}.
}
}
\date\today

\begin{document}
 
\maketitle

\begin{abstract} 
Assume $G$ is a graph, $(v_1,\ldots,v_k)$ is a sequence of distinct vertices of $G$, and  $(a_1,\ldots,a_k)$ is an integer sequence with $a_i \in \{1,2\}$. We say $G$ is \emph{$(a_1,\ldots,a_k)$-list extendable} (respectively, \emph{$(a_1,\ldots,a_k)$-AT extendable}) with respect to $(v_1,\ldots,v_k)$ if $G$ is $f$-choosable (respectively, $f$-AT), where $f(v_i)=a_i $ for $i \in \{1,\ldots, k\}$, and $f(v)=3$ for $v \in V(G) \setminus \{v_1,\ldots, v_k\}$. 
Hutchinson proved that if $G$ is an outerplanar graph, then $G$ is $(2,2)$-list extendable with respect to $(x,y)$ for any vertices $x,y$. 
We  strengthen this result and prove that if $G$ is a $K_4$-minor-free graph, then $G$ is $(2,2)$-AT extendable with respect to $(x,y)$ for any vertices $x,y$. Then we characterize all  triples $(x,y,z)$ of a $K_4$-minor-free graph $G$ for which $G$ is $(2,2,2)$-AT extendable (as well as $(2,2,2)$-list extendable) with respect to $(x,y,z)$. We also characterize the pairs $(x,y)$ of a $K_4$-minor-free graph $G$ for which $G$ is $(2,1)$-AT extendable (as well as $(2,1)$-list extendable) with respect to $(x,y)$.
Moreover, we characterize all triples $(x,y,z)$ of a 3-colorable graph $G$ with its maximum average degree less than $\frac{14}{5}$ for which $G$ is $(2,2,2)$-AT extendable with respect to $(x,y,z)$.
\end{abstract}
 
\section{Introduction}

Let $\mathbb{N}$ denote the set of positive integers, and for $k\in \mathbb{N}$, define $[k]$ to be the set $\{1, \ldots, k\}$.
Throughout the paper, we assume that $G$ is a simple graph unless stated otherwise.
Given $G$, let $V(G)$ and $E(G)$ denote its vertex set and edge set, respectively.
For an integer $k$, a vertex $v$ of $G$ is  a {\em $k$-vertex} (respectively, a {\em $k^+$-vertex} or a {\em $k^-$-vertex}) if $d_G(v)=k$  (respectively, $d_G(v)  \ge k$ or $d_G(v) \le k$). 
A {\em proper coloring} of $G$ is a function  $\phi: V(G) \to \mathbb{N}$ such that $\phi(u) \ne \phi(v)$ for each edge $uv$ of $G$. 
Given $k \in \mathbb{N}$, we say $G$ is {\em $k$-colorable} if $G$ has a proper coloring $\phi$ such that $\phi(V(G)) \subseteq [k]$.
A {\em list assignment} $L$ of $G$ is a function on $V(G)$  that assigns a list $L(v)\subseteq\mathbb N$ of \emph{available colors} to each vertex $v \in V(G)$.
Given a list assignment $L$ of $G$, an {\em $L$-coloring} $\varphi$ of $G$  is a proper coloring of $G$ such that $\varphi(v) \in L(v)$ for each vertex $v \in V(G)$. 
Let $\mathbb{N}^G$ be the set of all mappings $f: V(G) \to \mathbb{N}$. 
For a mapping $f\in\mathbb{N}^G$, we say $G$ is {\em $f$-choosable}
if $G$ has an $L$-coloring for every list assignment $L$ of $G$ for which $|L(v)| \ge f(v)$.
If $f$ is a constant function with value $k \in \mathbb{N}$, then we say an $f$-choosable graph $G$ is {\em $k$-choosable}.
The {\em list chromatic number} of $G$, denoted $\chi_l(G)$, is the minimum $k$ such that $G$ is $k$-choosable. List coloring of graphs has been studied extensively in the literature~\cite{TuzaSurvey}. 
A useful tool in the study of list coloring is the Combinatorial Nullstellensatz, and its associated Alon-Tarsi orientations of graphs.   

\begin{definition}
    \label{def-AT}  \rm
    Let $D$ be an orientation (of edges) of $G$. 
    An {\em Eulerian sub-digraph} of $D$ is a spanning sub-digraph $F$ of $D$ with $d_{F}^+(v) =  d_{F}^-(v)$ for every vertex $v \in V(G)$. Let $\EEu(D)$ (respectively, $\OEu(D)$) denote the set 
    of Eulerian sub-digraphs with an even (respectively, odd) number of arcs. 
    Let $$\diff(D)=|\EEu(D)|-|\OEu(D)|.$$ We say $D$ is an {\em Alon-Tarsi orientation} (shortened as {\em AT-orientation})  if $\diff(D)\neq 0$. 
    For a mapping $f\in\mathbb{N}^G$, we say $G$ is {\em $f$-Alon-Tarsi} (shortened as {\em $f$-AT}) if  $G$ has an AT-orientation $D$ with $d_D^+(v) \le  f(v)-1$ for each vertex $v \in V(G)$. 
    If $f$ is a constant function with value $k \in \mathbb{N}$, then we say an $f$-AT graph $G$ is $k$-AT.
    The {\em Alon-Tarsi number} of $G$, denoted  $AT(G)$,  is the minimum integer $k$ such that $G$ is $k$-AT.
\end{definition}

 \begin{AlonTarsi}[\cite{alon1992colorings}]\label{thm:AT}
If $G$ is $f$-AT, then $G$ is $f$-choosable. In particular, $\chi_l(G) \le AT(G)$.
\end{AlonTarsi}

The Alon-Tarsi number of a graph $G$ is not only an upper bound for both the {list chromatic} number of $G$ and the online {list chromatic} number of $G$~\cite{ZB-Book}, but also  a graph invariant of independent interest. 
A natural question is whether some upper bounds for the list chromatic number of graphs are also upper bounds for the Alon-Tarsi number. 
A classical result of  
Thomassen~\cite{thomassen1994every} says that every planar graph is $5$-choosable. This result was strengthened in~\cite{zhu2019alon}, where it was proved that every planar graph has Alon-Tarsi number at most $5$. 
 Indeed, Thomassen's classical result is stronger than the statement that every planar graph is $5$-choosable:   if $G$ is a plane graph with boundary cycle $(v_1v_2\ldots v_n)$, then   $G - v_1v_2$ is $f$-choosable, where $f(v_1)=f(v_2)=1$, $f(v_i)=3$ for $i \in \{3,\ldots, n\}$, and $f(v) =5$ for each interior vertex $v$.
 This stronger and more technical result is  useful in many cases, say for example in the study of choosability of locally planar graphs~\cite{DKM2008}. The result in~\cite{zhu2019alon} also says that $G-v_1v_2$ is $f$-AT for the same aforementioned function $f$.

In this paper, we are interested in list colorings and Alon-Tarsi orientations of $K_4$-minor-free graphs. It is well-known and easy to verify that $K_4$-minor-free graphs are $2$-degenerate, and hence has {list chromatic} number, as well as Alon-Tarsi number, at most $3$. We are interested in the problem whether a $K_4$-minor-free graph $G$ is $f$-choosable, or $f$-AT, for some $f \in \mathbb{N}^G$, 
 with $f(v) \le 3$ for every vertex $v$ of $G$, and $f(v) < 3$ for some vertices $v$ of $G$.

\begin{definition}
    \label{def-ext} \rm
    Assume $G$ is a graph, $(v_1,\ldots, v_k)$ is a $k$-tuple of distinct vertices of $G$, and $(a_1,\ldots, a_k)$ is a sequence of integers with $a_i \in [2]$ for $i \in [k]$. 
    Define $f\in\mathbb{N}^G$ as  $f(v_i)=a_i$ for $i \in [k]$ and $f(v)=3$ for $v \in V(G) \setminus \{v_1,\ldots, v_k\}$. If $G$ is $f$-choosable (respectively, $f$-AT), then we  say $G$ is {\em $(a_1,\ldots, a_k)$-list extendable} (respectively, {\em $(a_1,\ldots, a_k)$-AT extendable}) with respect to $(v_1,\ldots, v_k)$. An $f$-AT orientation of $G$ is called an {\em $(a_1,\ldots, a_k)$-AT orientation} of $G$ with respect to $(v_1,\ldots, v_k)$.
\end{definition}

Hutchinson~\cite{hutchinson2012list} first studied 
$f$-choosability of outerplanar graphs.
Hutchinson proved that all outerplanar graphs are $(2, 2)$-list extendable with respect to any pair of vertices $(x,y)$, and presented necessary and sufficient conditions for an outerplanar graph $G$ to be $(2,1)$-list extendable or $(1,1)$-list extendable with respect to $(x,y)$.

The results in this paper extend Hutchinson's results in two aspects: (1) we consider a more general class of graphs, from outerplanar graphs to $K_4$-minor-free graphs, or graphs with bounded maximum average degree and (2) prove stronger statements, from list extendability to AT extendability.
More precisely, we first extend Hutchinson's result to $K_4$-minor-free graphs, and strengthen the $(a,b)$-list extendable results to $(a,b)$-AT extendable  results. 
Then for an arbitrary $K_4$-minor-free graph $G$, we characterize all triples $(x,y,z)$ for which $G$ is $(2,2,2)$-list extendable, as well as $(2,2,2)$-AT extendable. Lastly, we discuss a similar question in the context of graphs with bounded maximum average degree.

\subsection{$(2,2)$-AT extendability of $K_4$-minor-free graphs}

To state our result, we need more definitions.
For $n\in\mathbb{N}$, let $D_n$ be the graph with    $$V(D_n) = \{ u_i, v_i,w_i: i \in [n]\} \cup \{u_0\}, \,\hfill\, 
E(D_n) = \{u_{i-1}v_i, u_{i-1}w_i, v_iw_i,v_i u_i, w_i  u_i : i \in [n]\}.$$ 
    For $n\in\mathbb N$, the graph $D_n$ is called a {\em chain of diamonds} (see   \Cref{fig:diamond}). 
Let $U(D_n)$ denote the set $\{u_0,u_1,\ldots,u_n\}$ of vertices of $D_n$.
\begin{figure}[h!]
    \centering
    \includegraphics{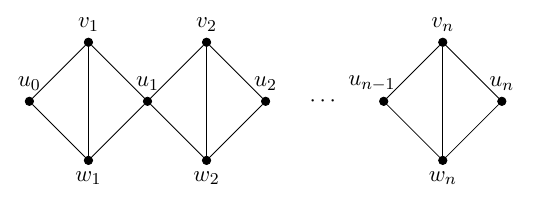}
    \caption{The graph $D_n$, a chain of diamonds}
    \label{fig:diamond}
\end{figure}  

\begin{definition} \rm
    Let $G$ be a graph. A set $X$ of distinct vertices of $G$ is said to be {\em connected by a chain of diamonds} if there is $n\in\mathbb N$ and a homomorphism $\varphi$ from a chain of diamonds $D_n$  to $G$ such that $X \subseteq  \varphi(U(D_n))$.   
\end{definition}

  \begin{observation}
      \label{obs:1} \rm
Observe that if $\phi$ is a proper $3$-coloring of a chain of diamonds $D_n$, then $\phi(u) = \phi(u')$ for all $u,u'\in U(D_n)$. Hence if a subset $X$ of vertices of a $3$-colorable graph $G$ is connected by a chain of diamonds, then all vertices in $X$ are colored by the same color in every proper 3-coloring of $G$. 
 \end{observation}

Our first result extends the work of~\cite{hutchinson2012list} to AT-extendability of $K_4$-minor-free graphs.

\begin{theorem}\label{thm:21}
Assume $G$ is  a $K_4$-minor-free graph and   $x,y$ are distinct vertices of $G$. Then $G$ is $(2,2)$-AT extendable with respect to $(x,y)$. Moreover, if $\{x,y\}$ is not connected by a chain of diamonds, then $G$ is $(2,1)$-AT extendable with respect to $(x,y)$.
\end{theorem}

The following corollary holds from \Cref{thm:21}.

\begin{corollary}
    \label{cor:21}
 Assume $G$ is a $K_4$-minor-free graph and $x,y$ are distinct vertices of $G$. Then   
the following are equivalent:
\begin{enumerate}[(1)]
    \item $G$ is $(2,1)$-AT extendable with respect to $(x,y)$.
    \item $G$ is $(2,1)$-list extendable with respect to $(x,y)$.
    \item $\{x,y\}$ is not connected by a chain of diamonds.
\end{enumerate}
\end{corollary}
\begin{proof} (3) $\Rightarrow$ (1) follows from \Cref{thm:21}, and  (1) $\Rightarrow$ (2) follows from the Alon-Tarsi Theorem.
To show (2) $\Rightarrow$ (3), assume $G$ is $(2,1)$-list extendable with respect to $(x,y)$.
Let $L(x) = \{1,2\}$, $L(y) = \{3\}$, and $L(v) = \{1,2,3\}$ for all $v \in V(G) \setminus \{x,y\}$.
Then $G$ has an $L$-coloring $f$, which is a proper 3-coloring of $G$ with $f(x) \ne f(y)$. By  \Cref{obs:1},  $\{x,y\}$ is not connected by a chain of diamonds.
\end{proof}

\subsection{$(2,2,2)$-AT extendability of $K_4$-minor-free graphs}

Our second result considers $(2,2,2)$-AT extendability of $K_4$-minor-free graphs and prove the following result.

\begin{definition}
    \label{def-feasible} \rm
    A set $X$ of three distinct vertices of $G$ is  {\em feasible} if $X$ is not connected by a chain of diamonds and there is a proper $3$-coloring $\phi$ of $G$ for which $|\phi(X)| \le 2$.
\end{definition}

\begin{theorem}\label{thm:k4minorfree}
    Assume $G$ is  a $K_4$-minor-free graph  and $x,y,z$ are distinct vertices of $G$. 
    If $\{x,y,z\}$ is feasible, then  $G$ is $(2,2,2)$-AT extendable with respect to $(x,y,z)$. 
\end{theorem}

As a corollary of \Cref{thm:k4minorfree}, the following holds.

\begin{corollary}\label{cor:222}
    Assume $G$ is  a $K_4$-minor-free graph  and $x,y,z$ are distinct vertices of $G$. Then 
 the following are equivalent:
\begin{enumerate}[(1)]
    \item $G$ is $(2,2,2)$-AT extendable with respect to $(x,y,z)$.
    \item $G$ is $(2,2,2)$-list extendable with respect to $(x,y,z)$.
 \item $\{x,y,z\}$ is feasible. 
\end{enumerate}
\end{corollary}
\begin{proof}
(3) $\Rightarrow$ (1) follows from  \Cref{thm:k4minorfree}, and  (1) $\Rightarrow$ (2) follows from the Alon-Tarsi Theorem.
To show (2) $\Rightarrow$ (3), assume $G$ is $(2,2,2)$-list extendable with respect to $(x,y,z)$. Let $L(x)=L(y)=L(z)=\{1,2\}$ and $L(v) = \{1,2,3\}$ for all $v \in V(G) \setminus \{x,y,z\}$. Then $G$ has an $L$-coloring $\varphi$, which is a proper $3$-coloring of $G$ such that $|\{\varphi(x), \varphi(y), \varphi(z)\}| \le 2$. 
 Let $L'(x) = \{1,2\}$, $L'(y) = \{1,3\}$, $L'(z) = \{2,3\}$, and $L'(v) = \{1,2,3\}$ for all $v \in V(G) \setminus \{x,y,z\}$. Again $G$ has an $L'$-coloring $\phi$, which is a proper $3$-coloring of $G$ that uses at least two colors on $\{x,y,z\}$. By   \Cref{obs:1}, $\{x,y,z\}$ is not connected by a chain of diamonds. Hence, the set $\{x, y, z\}$ is feasible. 
\end{proof}

Theorem 
\ref{thm:k4minorfree} is tight in the sense that there are $K_4$-minor-free graphs $G$ such that $G$ is not $(2,2,1)$-list extendable with respect to $(x,y,z)$ for any distinct $x,y,z \in V(G)$. Indeed, any $K_4$-minor-free graph with a unique proper $3$-coloring has this property.

Here is a sketch of proof.
Let $G$ be a $K_4$-minor-free graph with a unique proper $3$-coloring $\phi$. 
If $ \phi(x) =  \phi(z)$, then let $L(x) = L(y) = \{1,2\}$, $L(z) = \{3\}$, and $L(v) = \{1,2,3\}$ for any other vertex $v$. It is easy to see that $G$ is not $L$-colorable, as any proper 3-coloring of $G$ colors $x, z$ with the same color. The case where $\phi(y) = \phi(z)$ is symmetric. 
    
Assume $\phi(z) \not\in \{\phi(x), \phi(y)\} $. If $\phi(x) \ne \phi(y)$, then let $L(x)=L(y)  = \{1,2\}$, $L(z) = \{1\}$, and $L(v) = \{1,2,3\}$ for any other vertex $v$. Then $G$ is not $L$-colorable. 
If $\phi(x) = \phi(y)$, then let $L(x)=\{1,2\}, L(y) = \{1,3\}, L(z) = \{1\}$, and $L(v) = \{1,2,3\}$ for any other vertex $v$. Again $G$ is not $L$-colorable.

In particular, $2$-trees are $K_4$-minor-free graphs that have a unique proper $3$-coloring. So they are not $(2,2,1)$-list extendable with respect to $(x,y,z)$ for any distinct vertices $x,y,z$. 

Despite the above observation, if a $K_4$-minor-free graph is triangle-free, then we have the following result.

\begin{theorem}\label{thm:tranglefree}
If $G$ is a triangle-free $K_4$-minor-free graph, then $G$ is $(2,2,1)$-AT extendable with respect to $(x,y,z)$ for every $x,y,z \in V(G)$.
\end{theorem}
\begin{proof}
Let $G$ be a triangle-free $K_4$-minor-free graph.
If $|V(G)| \le 3$, then \Cref{thm:tranglefree} holds.
Thus, we may assume that $|V(G)| \ge 4$.
Note that $G$ has at least four $2^-$-vertices since it is a triangle-free $K_4$-minor-free graph.

Suppose to the contrary that there is a graph $G$ and its vertices $x,y,z \in V(G)$ such that $G$ is not $(2,2,1)$-AT extendable with respect to $(x,y,z)$.
Let $G$ be the minimal graph with this property with respect to $|V(G)|$.
Let $w\in V(G)\setminus\{x,y,z\}$ such that $w$ is a $2^-$-vertex.
By the minimality of $G$, $G-w$ has a $(2,2,1)$-AT orientation $D'$ with respect to $(x,y,z)$.
Then we obtain an orientation $D$ of $G$ by starting with $D'$ and then orienting the edges incident with $w$ so that $d_D^-(w)=0$.
Then $D$ is a $(2,2,1)$-AT orientation of $G$, which is a contradiction.
\end{proof}

\subsection{$(2,2,2)$-AT extendability of graphs with bounded  maximum average degree}

\begin{definition} \rm 
A set of three distinct vertices $\{x,y,z\}$ of $G$ is {\em blocked} if either for every proper 3-coloring $\phi$ of $G$, $|\phi(\{x,y,z\})|=3$ or for every proper 3-coloring $\phi$ of $G$, $|\phi(\{x,y,z\})|=1$.
\end{definition}

Note that $G$ is not $(2,2,2)$-list extendable with respect to $(x,y,z)$ if $\{x,y,z\}$ is blocked. 
By using  \cref{cor:222}, it is easy to check that if $G$ is $K_4$-minor-free and $\{x,y,z\}$ is non-blocked, then $G$ is $(2,2,2)$-list extendable with respect to $(x,y,z)$. \Cref{fig:not222} shows that the condition that $G$ be $K_4$-minor-free cannot be simply dropped. Nevertheless, we prove that if $G$ has ${\rm mad}(G) < \frac {14}{5}$, then $G$ is $(2,2,2)$-list extendable with respect to $(x,y,z)$, provided that  $\{x,y,z\}$ is non-blocked;
recall that the \emph{maximum average degree} of $G$, denoted ${\rm mad}(G)$, is defined as $\max_{H \subseteq G} {2|E(H)| \over |V(H)|}$. Note that the graph in \Cref{fig:not222} has ${\rm mad}(G)=\frac {14}{5}$.

\begin{theorem}\label{thm:mad}
If $G$ is a graph with $\mad(G)<{14\over 5}$, then $G$ is $(2,2,2)$-AT extendable with respect to $(x,y,z)$ for every $\{x,y,z\}$ that is non-blocked.
\end{theorem}

\begin{figure}
\centering
\includegraphics[]{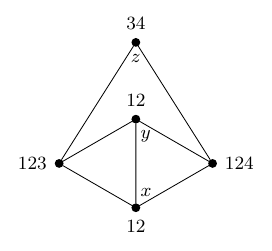}
    \caption{A graph with maximum average degree $\frac{14}{5}$ that is not $(2,2,2)$-list extendable with respect to $(x,y,z)$ that is non-blocked}
    \label{fig:not222}
\end{figure}

\section{Preliminaries}\label{sec:pre}

 If $u$ is a cut-vertex of a graph $G$, and $G_1,G_2$ are two induced connected subgraphs of $G$ with $V(G_1) \cap V(G_2) =\{u\}$ and  $V(G_1) \cup V(G_2) =V(G)$,
then we say $u$ {\em separates $G$ into  $G_1$ and $G_2$}.
For an orientation $D$ of $G$, let $A(D)$ denote its arc set.
 
\begin{lemma}\label{lem:vxcut}
For a graph $G$, let $u$ be a cut-vertex (of $G$) that separates $G_1$ and $G_2$.
For an orientation $D$ of $G$, let $D_i$ be the orientation of $G_i$ obtained by restricting $D$ onto $G_i$ for $i \in [2]$.
Then $\diff(D) = \diff(D_1) \times \diff(D_2)$.
\end{lemma}
\begin{proof}
For an orientation $D$ of $G$, let $D_i$ be the orientation of $G_i$ obtained by restricting $D$ onto $G_i$ for $i \in [2]$.
Let $F$ be an Eulerian sub-digraph of $D$, and let $F_i$ be the sub-digraph of $D_i$ obtained by restricting $F$ onto $D_i$ for $i \in [2]$.
Note that $\sum_{v \in V(F_1)}d^+_{F_1}(v) = \sum_{v \in V(F_1)}d^-_{F_1}(v)$, and $d^+_{F_1}(v) = d^-_{F_1}(v)$ for all $v \in V(F_1) \setminus \{u\}$.
Thus, $d^+_{F_1}(u) = d^-_{F_1}(u)$ and therefore $F_1$ is an Eulerian sub-digraph of $D_1$.
Similarly, $F_2$ is an Eulerian sub-digraph of $D_2$.
So $F$ is the   disjoint (with respect to arcs) union of $F_1$ and $F_2$. Conversely, for any Eulerian sub-digraph $F_1$ of $D_1$, and any  Eulerian sub-digraph $F_2$ of $D_2$, $F=F_1 \cup F_2$ is an Eulerian sub-digraph of $D$.
Note that $|A(F)|$ is even if $|A(F_1)|$ and $|A(F_2)|$ have the same parity, and $|A(F)|$ is odd if $|A
(F_1)|$ and $|A(F_2)|$ have different parities. 
Hence $\diff(D) = \diff(D_1) \times \diff(D_2)$.
\end{proof}

\begin{lemma}
    \label{lem:triangle}
    Assume $G$ is a graph,  $[u_1u_2u_3]$ is a triangle,
    $d_G(u_1)=2$, and $d_G(u_2)=3$ with $N_G(u_2)=\{u_1,u_3, u_4\}$. Let $D$ be an orientation of   
    $G$ in which the edges incident with $u_1$ or $u_2$ 
    are oriented as $(u_1,u_3), (u_2, u_1), (u_2, u_3), (u_4, u_2)$. 
    Let $D'= D-\{u_1, u_2\}$. Then $${\rm diff}(D) = {\rm diff}(D').$$ 
    In particular, $D$ is an AT-orientation if and only if $D'$ is an AT-orientation.
\end{lemma}
\begin{proof}
    Each Eulerian sub-digraph of $D'$ is an Eulerian sub-digraph of $D$ (with $u_1,u_2$ being isolated vertices).  On the other hand,  if $F$ is an Eulerian sub-digraph of $D$ but 
    $F-\{u_1,u_2\}$ is not an Eulerian sub-digraph of $D'$, then  $(u_4, u_2) \in F$, and exactly one of $P_1=(u_2,u_3)$ and $P_2 = (u_2,u_1,u_3)$ is contained in $F$. 
    For $i \in [2]$, let $\mathcal{E}_i$ be the Eulerian sub-digraphs of $D$ containing $P_i$. If $F \in \mathcal{E}_i$, then $F'=(F-P_i) \cup P_{3-i} \in \mathcal{E}_{3-i}$, and $F$ and $F'$ have different parities.   So the Eulerian sub-digraphs of $D$ that are not Eulerian sub-digraphs of $D'$ contributes $0$ to the difference ${\rm diff}(D)$ of $D$. Hence ${\rm diff}(D) = {\rm diff}(D').$
\end{proof}
 
 \begin{lemma}
     \label{lem-feasible}
     Assume $G$ is a $K_4$-minor-free graph and $X$ is a set of three vertices of $G$. If  there is a proper $3$-coloring $\phi$ of $G$ such that $|\phi(X)|=2$, then $X$ is feasible. 
 \end{lemma}
 \begin{proof}
       Since $|X|=3$ and $|\phi(X)|=2$, by  \Cref{obs:1}, $X$ is not connected by a chain of diamonds. As $|\phi(X)|=2$,  $X$ is feasible. 
 \end{proof}

 \begin{corollary}
     \label{cor:xx'}
    Assume $G$ is a $K_4$-minor-free graph, $xx'$ is an edge of $G$, and $y,z \not\in \{x, x'\}$. Then at least one of the sets $\{x,y,z\}$ and $\{x',y,z\}$ is feasible.
 \end{corollary}
 \begin{proof}
      Assume  $xx'$ is an edge of $G$ and $\phi$ is a proper 3-coloring of $G$. Since $\phi(x) \ne \phi(x')$, for any vertices $y,z \not\in \{x,x'\}$, at least one of $\phi(\{x,y,z\})$ and $\phi(\{x',y,z\})$ has size 2. Hence, by \cref{lem-feasible}, at least one of $\{x,y,z\}$ and $\{x',y,z\}$ is feasible.
 \end{proof}

\begin{observation}\label{obs:non-adjacent} \rm
 Assume $G$ is a connected $K_4$-minor-free graph with minimum degree $2$ that is not a cycle.    
 It is well-known~\cite{wald1983steiner} that   $G$ has a plane embedding where  two faces of $G$ each has an incident $2$-vertex. 
 In particular, $G$ has at least two non-adjacent $2$-vertices.  
\end{observation}

 \begin{definition}
     \label{def:genuine} \rm
     Assume $G$ is a connected $K_4$-minor-free graph with minimum degree $2$ that is not a cycle.  
 Assume $v$ is a $2$-vertex on a triangle $[vu_1u_2]$, and if $u_i$ is a $3$-vertex, then let $w_i\in N_G(u_i)\setminus\{v,u_{3-i}\}$.
We say $v$ is {\em genuine} if (1) or (2) holds:
\begin{enumerate}[(1)]
\item
$d_G(u_i)=3$ and $w_iu_{3-i}$ is an edge of $G$ for some $i\in[2]$.
\item 
$d_G(u_1)=d_G(u_2)=3$ and $w_1w_2$ is an edge of $G$.
\end{enumerate}
 \end{definition}

\Cref{fig:genuine} is an illustration of genuine vertices,  where a vertex represented by a square in the figure indicates that its degree is unspecified.

\begin{figure}
    \centering
    \includegraphics{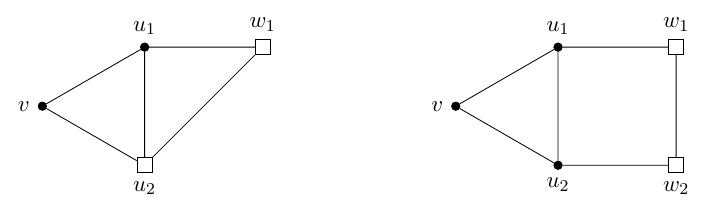}
    \caption{Figures for a genuine vertex $v$}
    \label{fig:genuine}
\end{figure}

\begin{lemma}
    \label{lem:genuine}
     Assume $G$ is a connected $K_4$-minor-free graph with minimum degree $2$ that is not a cycle, and $G$ has at most three $2$-vertices.
     If $G$ has only two $2$-vertices, then both  are   genuine $2$-vertices. If $G$ has 
     exactly three $2$-vertices, then at least one of them is a genuine $2$-vertex. 
\end{lemma}
\begin{proof}
    For each non-genuine $2$-vertex $v$ whose neighbors $u_1,u_2$ are $3^+$-vertices, if $u_1u_2$ is not an edge of $G$, then contract the edge $vu_1$. If $u_1u_2$ is an edge of $G$, then we do the following operation:
    \begin{enumerate}
        \item If   both $u_1,u_2$ are $4^+$-vertices, then delete $v$.
        \item If both $u_1,u_2$ are $3$-vertices, where $w_i\in N_G(u_i)\setminus\{v,u_{3-i}\}$ for $i \in [2]$,  $w_1\ne w_2$,  and $w_1w_2 $ is not an edge of $G$, then contract all the edges $vu_1,vu_2, u_1u_2, u_1w_1$. 
        \item If $u_1$ is a $3$-vertex and $u_2$ is a $4^+$-vertex, where $w_1\in N_G(u_1)\setminus\{v,u_2\}$, and $w_1u_2$ is not an edge of $G$, then contract all the edges $vu_1,vu_2, u_1u_2$.
    \end{enumerate}
    
    We denote by $G'$ the resulting graph. It follows from the construction that each non-genuine $2$-vertex of $G$ with two $3^+$-neighbors is not a vertex of $G'$, and no new $2$-vertices are created.  In other words,  every vertex of $G'$ has degree at least $3$, except genuine $2$-vertices of $G$ or $2$-vertices of $G$ with a $2$-neighbor in $G$. 
    Note that $G'$ is also a connected $K_4$-minor-free graph with minimum degree at least $2$ that is not a cycle, so that $G'$ also has at least two non-adjacent $2$-vertices by \cref{obs:non-adjacent}.

    If $G$ has only two $2$-vertices, which are non-adjacent by \cref{obs:non-adjacent}, and at least one them is not genuine, then $G'$ has at most one $2$-vertex, which is a contradiction. 
    Thus, if $G$ has only two $2$-vertices, then both are genuine $2$-vertices.
    
   If $G$ has exactly three $2$-vertices, and none of them are genuine $2$-vertices, then every $2$-vertex in $G'$ is a $2$-vertex that is adjacent to a $2$-vertex in $G$.
   Since $G$ has exactly three $2$-vertices, every $2$-vertex in $G'$ should be on the same face of $G'$, which is a contradiction to \cref{obs:non-adjacent}.
   Thus, if $G$ has exactly three $2$-vertices, then at least one of them is a genuine $2$-vertex.
\end{proof}

A {\it Gallai tree} is a graph in which every block is a complete graph or an odd cycle.
The following theorem is known as the {\it  degree-AT theorem}.

\begin{theorem}[The degree-AT theorem\cite{hladky2010brooks}]\label{thm:degreeAT}
Let $G$ be a connected graph.
If $G$ is not a Gallai tree, then $G$ has an AT orientation $D$ such that $d_D^-(v) \ge 1$ for each $v \in V(G)$.
\end{theorem}

 \section{Proof of  \Cref{thm:21}}

Assume $G$ is a $K_4$-minor-free graph and $x,y$ are distinct vertices of $G$. We prove by induction on the number of vertices of $G$ that $G$ is $(2,2)$-AT extendable with respect to $(x,y)$, and if $\{x,y\}$ is not connected by a chain of diamonds, then $G$ is $(2,1)$-AT extendable with respect to $(x,y)$. 

By induction, we may assume that $G$ is connected.
If $G$ is a subgraph of a cycle, then it is easily checked that for any distinct vertices $x, y$ of $G$, $G$ is $(2,1)$-AT extendable with respect to $(x,y)$. 
Thus, assume $G$ is not a subgraph of a cycle, and this implies that $G$ has at least four vertices. 

If $d_G(v) \le 2$ for some vertex $v \not\in \{x,y\}$, then by induction $G'=G-v$ has a $(2,2)$-AT orientation (if $\{x,y\}$ is not connected by a chain of diamonds, then $(2,1)$-AT orientation) $D'$ with respect to $(x,y)$. 
By orienting the edges incident with $v$ as out-arcs of $v$, we obtain a $(2,2)$-AT orientation (if $\{x,y\}$ is not connected by a chain of diamonds, then $(2,1)$-AT orientation) of $G$ with respect to $(x,y)$.
Thus, we may assume that all vertices other than $x,y$ are $3^+$-vertices.

If $d_G(x) =1$, then 
$G-x$ has a $(2,1)$-AT orientation with respect to $(z,y)$, where $z$ is a vertex of $G-x$ such that $\{z,y\}$ is not connected by a chain of diamonds.
This can be extended to a $(2,1)$-AT orientation of $G$ with respect to $(x,y)$ by orienting the edge incident with $x$ as an out-arc of $x$. 
Thus, we may assume that $x$ is a $2^+$-vertex.

Suppose $d_G(y) = 1$, and $N_G(y)= \{z\}$.
If $x = z$, then $G-y$ has a $(2,1)$-AT orientation with respect to $(w,x)$, where $w$ is a vertex of $G-y$ such that $\{w,x\}$ is not connected by a chain of diamonds. 
This can be extended to a $(2,1)$-AT orientation of $G$ with respect to $(x,y)$ by orienting the edge incident with $y$ as an in-arc of $y$.
If $x \neq z$, then $G-y$ has a $(2,2)$-AT orientation with respect to $(z,x)$, which can be extended to a $(2,1)$-AT orientation of $G$ with respect to $(x,y)$ by orienting the edge incident with $y$ as an in-arc of $y$.
Thus, we may assume that $y$ is a $2^+$-vertex.

Thus $G$ has minimum degree 2, and all vertices other than $x,y$ are  $3^+$-vertices. So both $x,y$ are genuine $2$-vertices of $G$ by Lemma~\ref{lem:genuine}.

As $x$ is  a genuine $2$-vertex of $G$, we may assume that $[xu_1u_2]$ is a triangle,   $d_G(u_1)=3$ and $N_G(u_1)=\{x,u_2,w_1\}$.    

\medskip
\noindent
{\bf Case 1.}
$u_2w_1$ is an edge of $G$.
\medskip

As $[xu_1u_2]$ is a triangle,   $\{x,w_1\}$ is connected by a chain of diamonds.  
If $w_1 \ne y$, then by induction $G'=G-\{x,u_1\}$ has a $(2,2)$-AT orientation $D'$ with respect to $(w_1,y)$. 
Add the arcs $(w_1,u_1),(u_1,x),(x,u_2),(u_1,u_2)$ to $D'$ to obtain an orientation $D$ of $G$. By \Cref{lem:triangle}, ${\rm diff}(D)={\rm diff}(D')$. Hence $D$ is a 
$(2,2)$-AT orientation of $G$ with respect to $(x,y)$. If $\{x,y\}$ is not connected by a chain of diamonds, then $\{w_1,y\}$ is not connected by a chain of diamonds. Hence we may assume that $D'$ is a $(2,1)$-AT orientation of $G'$ with respect to $(w_1,y)$, and therefore  $D$ is a 
$(2,1)$-AT orientation of $G$ with respect to $(x,y)$.

If $w_1=y$, then $\{u_2,y\}$ is not connected by a chain of diamonds. 
So by induction, $G'=G-\{x,u_1\}$ has a $(2,1)$-AT orientation $D'$ with respect to $(u_2,y)$. Add the arcs $(y,u_1)$, $(u_1,x)$, $(x,u_2)$, $(u_1,u_2)$ to $D'$ to obtain a $(2,2)$-AT orientation $D$ of $G$ with respect to $(x,y)$.  

\medskip
\noindent
{\bf Case 2.}
$u_2w_1$ is not an edge of $G$.
\medskip

By the definition of a genuine $2$-vertex,  $d_G(u_2)=3$, $N_G(u_2) = \{x,u_1,w_2\}$, and $w_1w_2$ is an edge of $G$. Then either $\{w_1,y\}$ or $\{w_2, y\}$ is not connected by a chain of diamonds (note that if $y=w_2$, then $\{w_1,y\}$ is not connected by a chain of diamonds). By symmetry, we may assume that $\{w_1, y\}$ is not connected by a chain of diamonds. Then by induction $G'=G-\{x,u_1\}$ has a $(2,1)$-AT orientation $D'$ with respect to $(w_1,y)$. 
Add the arcs $(w_1,u_1),(u_1,x),(x,u_2),(u_1,u_2)$ to $D'$. 
By  Lemma~\ref{lem:triangle}, the resulting orientation is a  $(2,1)$-AT   orientation  of $G$ with respect to $(x,y)$.   This completes the proof of  \Cref{thm:21}.

\section{Proof of  \Cref{thm:k4minorfree}}\label{sec:k4minorfree}

Assume $G$ is a $K_4$-minor-free graph and  $x,y,z \in V(G)$ are distinct vertices. Suppose to the contrary that $\{x,y,z\}$ is a feasible set, but $G$ is not $(2,2,2)$-AT extendable with respect to $(x,y,z)$.   
Let $G$ be such a graph with the minimum number of vertices. Then $G$ is connected, and for the same reason as in the proof of  \Cref{thm:21}, $G$ is not a subgraph of a cycle, and $d_G(x), d_G(y), d_G(z) \ge 2$,  and $d_G(v) \ge 3$ for each $v \in V(G) \setminus \{x,y,z\}$.   

Since $G$ is not a subgraph of a cycle, $G$ has at least four vertices. 
By \Cref{lem:genuine}, we may assume $x$ is a genuine $2$-vertex of $G$, and $[xu_1u_2]$ is a triangle, $d_G(u_1)=3$ and $N_G(u_1)=\{x,u_2,w_1\}$.  

\medskip
\noindent
{\bf Case 1.}
$u_2w_1$ is an edge of $G$.
\medskip

Then $\{x,w_1\}$ is connected by a chain of diamonds so that $\{w_1,y,z\}$ is feasible if $w_1 \not\in \{y,z\}$.

Suppose $u_1 \in \{y,z\}$. Say $u_1 = z$.
By \Cref{lem:genuine}, $y$ is a genuine $2$-vertex of $G$.
If $y \neq w_1$, then we may switch $x$ and $y$ so that $u_1 \neq z$.
Note that after switching, it is possible that $w_1 = z$.
If $w_1 = z$, then since $\{y,z\}$ is not connected by a chain of diamonds, by \Cref{thm:21}, $G' = G=\{x,u_1\}$ has a $(2,1)$-AT orientation $D'$ with respect to $(y,z)$.
By \Cref{lem:triangle}, $D'$ can be extended to a $(2,2,2)$-AT orientation $D$ of $G$, by adding the arcs $(z,u_1), (u_1,x), (x, u_2), (u_1,u_2)$ to $D'$, which is a contradiction.
Thus, $w_1 \neq z$.
In this case, $\{w_1, y, z\}$ is feasible in $G'=G-\{x,u_1\}$, and by minimality, $G'$ has a $(2,2,2)$-AT orientation $D'$ with respect to $(w_1,y,z)$. 
Add the arcs $(w_1,u_1)$, $(u_1,x)$, $(x,u_2)$, $(u_1,u_2)$ to obtain an orientation $D$ of $G$. 
By \Cref{lem:triangle}, ${\rm diff}(D)={\rm diff}(D')$. 
Hence $D$ is a $(2,2,2)$-AT orientation of $G$ with respect to $(x,y,z)$.  

If $y = w_1$, then by \Cref{thm:21}, $G' = G - \{x,y,z\}$ has a $(1)$-AT orientation $D'$ with respect to $(u_2)$.
Add the arcs $(u_2,y)$, $(y, z)$, $(z,u_2)$, $(u_2,x)$, $(x,z)$ to obtain an orientation $D$ of $G$.
Let $G''$ be a subgraph of $G$ induced by $\{x,y,z,u_2\}$, and $D''$ be an orientation of $G''$ obtained by restricting $D$ onto $G''$. 
By \Cref{lem:vxcut}, ${\rm diff}(D) = {\rm diff}(D') \times {\rm diff}(D'') \neq 0$.
Hence $D$ is a $(2,2,2)$-AT orientation of $G$ with respect to $(x,y,z)$.  

Suppose $u_1 \not\in \{y,z\}$.
If $w_1 \not\in \{y,z\}$, then 
$\{w_1,y,z\}$ is feasible in $G' = G - \{x,u_1\}$, and by minimality, 
$G'$ has a $(2,2,2)$-AT orientation $D'$ with respect to $(w_1,y,z)$. Add the arcs $(w_1,u_1)$, $(u_1,x)$, $(x,u_2)$, $(u_1,u_2)$ to  obtain an orientation $D$ of $G$. 
By  \Cref{lem:triangle}, ${\rm diff}(D)={\rm diff}(D')$. Hence $D$ is a 
$(2,2,2)$-AT orientation of $G$ with respect to $(x,y,z)$.  

If $w_1 \in \{y,z\}$, say $w_1=y$, then $\{y,z\}$ is not connected by a chain of diamonds. By  \Cref{thm:21}, $G'=G-\{x,u_1\}$ has a $(2,1)$-AT orientation $D'$ with respect to $(z,y)$. 
Add the arcs $(y,u_1),(u_1,x),(x,u_2),(u_1,u_2)$ to $D'$ to obtain $D$.
By  \Cref{lem:triangle}, $D$
is a $(2,2,2)$-AT   orientation  of $G$ with respect to $(x,y,z)$.

\medskip
\noindent
{\bf Case 2.}
$u_2w_1$ is not an edge of $G$.
\medskip

Then $d_G(u_2)=3$, $N_G(u_2) = \{x,u_1,w_2\}$ and $w_1w_2$ is an edge of $G$. 
If $\{u_1,u_2\} \cap \{y,z\} \neq \emptyset$, say $u_1 = z$, then by \Cref{lem:genuine}, $y$ is a genuine $2$-vertex, and we can switch $x$ and $y$ so that $\{u_1,u_2\} \cap \{y,z\} = \emptyset$.
Thus, we may assume that $\{u_1,u_2\} \cap \{y,z\} = \emptyset$.

If $\{y,z\} = \{w_1,w_2\}$, say $y=w_1,z=w_2$,
then by Theorem 
\ref{thm:21},   $G'=G-\{x,u_1, u_2\}$ has a $(2,1)$-AT orientation $D'$ with respect to $(z,y)$. 
Add the arcs $(y,u_1)$, $(u_1,x)$, $(x,u_2)$, $(u_1,u_2)$, $(u_2,z)$ to $D'$ to obtain a $(2,2,2)$-AT   orientation $D$ of $G$ with respect to $(x,y,z)$.  

If $ y,z  \notin  \{w_1,w_2\}$, then by  \Cref{cor:xx'}, $\{w_1,y,z\}$ or $\{w_2,y,z\}$ is feasible. By symmetry, we may assume that $\{w_1,y,z\}$ is feasible. Then  by minimality $G'=G-\{x,u_1,u_2\}$ has a $(2,2,2)$-AT orientation $D'$ with respect to $(w_1,y,z)$. By adding the arcs $(w_1,u_1)$, $(u_1,x)$, $(x,u_2)$, $(u_1,u_2)$, $(u_2,w_2)$, we obtain a $(2,2,2)$-AT orientation $D$ of $G$ with respect to $(x,y,z)$.  

Assume $|\{y,z\} \cap \{w_1,w_2\}|=1$, say $y=w_1$ and $z \notin \{w_1,w_2\}$. If $\{y,z\}$ is not connected by a chain of diamonds, then by  \Cref{thm:21}, $G'=G-\{x,u_1,u_2\}$ has a $(2,1)$-AT orientation $D'$ with respect to $(z,y)$. Add the arcs $(y,u_1),(u_1,x),(x,u_2),(u_1,u_2), (u_2,w_2)$ to $D'$ to obtain a $(2,2,2)$-AT   orientation $D$ of $G$ with respect to $(x,y,z)$.

If $\{y,z\}$ is connected by a chain of diamonds, then for any proper 3-coloring $\phi$ of $G'=G-\{x,u_1,u_2\}$, $\phi(y)=\phi(z)$ and hence $|\{\phi(w_2), \phi(y), \phi(z)\}|=2$. So $\{w_2, y,z\}$ is feasible, and by minimality, $G'$ has a $(2,2,2)$-AT orientation $D'$ with respect to $(w_2,y,z)$. 
Add the arcs $(w_2,u_2), (u_2,x), (x, u_1), (u_2,u_1), (u_1,y)$ to $D'$.
By \Cref{lem:triangle}, the resulting orientation $D$ is a $(2,2,2)$-AT orientation of $G$ with respect to $(x,y,z)$. This completes the proof of \Cref{thm:k4minorfree}.

\section{Proof of \Cref{thm:mad}}

Suppose to the contrary that there is a graph $G$ with $\mad(G) < \frac{14}{5}$, and  $G$ is not $(2,2,2)$-AT extendable with respect to a non-blocked triple $(x,y,z)$.
Let $G$ be a minimal graph with this property with respect to $|V(G)|$.
By \cref{cor:222}, $G$ has a $K_4$-minor.

\begin{claim}\label{clm:mad_degree}
Every vertex $v\in V(G)\setminus\{x,y,z\}$ is a $3^+$-vertex in $G$, and $x,y,z$ are $2^+$-vertices in $G$.
\end{claim}
\begin{proof}
Suppose to the contrary that $v$ is a $2^-$-vertex in $V(G)\setminus\{x,y,z\}$. 
By minimality of $G$, $G-v$ has a $(2,2,2)$-AT orientation with respect to $(x,y,z)$. 
By orienting the edges incident with $v$ as out-arcs of $v$, we obtain a $(2,2,2)$-AT orientation of $G$ with respect to $(x,y,z)$, which is a contradiction.
Thus, every vertex in $V(G) \setminus \{x,y,z\}$ is a $3^+$-vertex in $G$.

Suppose to the contrary that $v$ is a $1^-$-vertex in $\{x, y,z\}$. 
Without loss of generality, let $v = x$.
Let $x'$ be a vertex such that $\{x',y,z\}$ is non-blocked in $G'=G-x$. 
By minimality of $G$, $G'$ has a $(2,2,2)$-AT orientation with respect to $(x',y,z)$.
By orienting the edge incident with $x$ as an out-arc of $x$, we obtain a $(2,2,2)$-AT orientation of $G$ with respect to $(x,y,z)$, which is a contradiction.
Thus, $x,y,z$ are $2^+$-vertices in $G$.
\end{proof}

Let $n_i$ and $n_i^+$  be the number of $i$-vertices and $i^+$-vertices, respectively, in $G$. By Claim~\ref{clm:mad_degree}, $n_0=n_1=0$ and so 
\[ 0>5 \mad(G)|V(G)| -14|V(G)| \geq \sum_{v \in V(G)} (5d_G(v)-14) \ge  -4n_2 + n_3 +6 n_4+ 11 n_5^+.\]
Let $B=\{w\in\{x,y,z\}: w\text{ is a $2$-vertex}\}$.
By Claim~\ref{clm:mad_degree}, 
\[ 12\ge 4|B| =4n_2> n_3 +6 n_4+ 11 n_5^+ \ge n_3^+= |V(G)\setminus B|\ge 5-|B|.\]
Thus $|B| \ge 2$ and  $n_5^+=0$. Then the vertices with odd degree  are all $3$-vertices, which implies that $n_3$ is even and so 
\begin{eqnarray} \label{eq:B}
&&12\ge 4|B| \ge  n_3 + 6n_4 + 2.\end{eqnarray}
This also implies that $n_4\le 1$.

In the following, we will find an orientation $D$ of $G$ such that $D$ is a $(2,2,2)$-AT orientation of $G$ with respect to $(x,y,z)$, that is, $\Delta^+(D)\le 2$, $x,y,z$ have at most one out-arc in $D$, and $\diff(D)\neq 0$. Let $H^*$ be the graph obtained from $G$ by contracting an edge incident with a 2-vertex one by one. 
Note that $H^*$ may have multiple edges or loops.
Since $G$ has a $K_4$-minor, $H^*$ also has a $K_4$-minor, and therefore $|V(H^*)|\ge 4$.
Given an orientation $D^*$ of $H^*$, obtaining an orientation $D$ of $G$ by the following  is called a {\em recovering process}: For $(u, w)\in A(D^*)$,
\begin{enumerate}[(i)]
\item if $u\neq w$, then an edge $uw$ of $H^*$ corresponds to a path $v_1\ldots v_k$ in $G$ with internal $2$-vertices where $v_1=u$ and $v_k=w$ and so   let $ (v_i,v_{i+1})\in A(D)$ 
for all $i \in [k-1]$;
\item if $u=w$, then the loop $uw$ of $H^*$ corresponds to a cycle $v_1\ldots v_k$ with internal $2$-vertices in $G$ where $v_1=v_k=u$, then let $(u,v_{k-1})\in A(D)$ and 
let $ (v_i, v_{i+1}) \in A(D)$ for all $i \in [k-2]$.
\end{enumerate}

\medskip
\noindent
\textbf{Case 1.} $|B|=3$, that is, $x,y,z$ are all $2$-vertices.

\medskip

If all vertices in $V(G)\setminus \{x,y,z\}$ are $3$-vertices, then since $G$ is not $(2,2,2)$-AT extendable with respect to $(x,y,z)$, by the degree-AT theorem (\Cref{thm:degreeAT}), $G$ is a Gallai tree such that every block must be an odd cycle or a $K_2$. This is a contradiction to the fact that $G$ has a $K_4$-minor. Thus $G$ has a unique $4$-vertex. 
By \eqref{eq:B}, $n_3\le 4$.
Since $|V(H^*)| \ge 4$, $n_3=4$, and so $|V(G)|= 8$. 
Then $G$ is a graph with degree sequence $(4,3,3,3,3,2,2,2)$, and $H^*$ is a graph with degree sequence $(4,3,3,3,3)$. 
Since $H^*$ has a $K_4$-minor, there are two vertices $v_1$ and $v_2$ of $H^*$ such that $V(H^*)\setminus\{v_1,v_2\}$ is a triangle and each vertex in $V(H^*)\setminus\{v_1,v_2\}$ has a neighbor in $\{v_1,v_2\}$. Let $V(H^*)\setminus\{v_1,v_2\}=\{v_3,v_4,v_5\}$. 
By the degree constraint, a loop or a multiple edge is incident only with $v_1$ or $v_2$ if it exists. 
There are three possible graphs for $H^*$, and for each case, we give an orientation $D^*$ of $H^*$ as in \Cref{fig:43333222_ori}. 

\begin{figure}[h!]
    \centering    
    \includegraphics{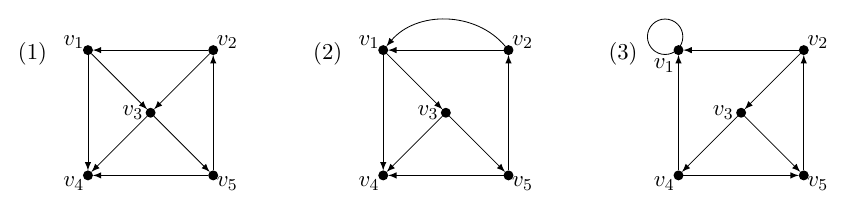}
    \caption{Orientations $D^*$ of $H^*$ when $G$ has degree sequence $(4,3,3,3,3,2,2,2)$}
    \label{fig:43333222_ori}
\end{figure}

Let $D$ be an orientation of $G$ obtained from $D^*$ by the recovering process. 
Note that $D^*-v_2v_5$ is acyclic, any nonempty Eulerian sub-digraph of $D^*$ contains the arc $v_2v_5$, and so there are exactly two nonempty Eulerian sub-digraphs in $D^*$. Thus $D$ also has an odd number of 
 Eulerian sub-digraphs and so $\diff(D)\neq 0$.

\medskip
\noindent
{\bf Case 2.} $|B|=2$
\medskip

We may assume that $x,y$ are $2$-vertices, and $z$ is a $3^+$-vertex in $G$. Since $n_3 + n_4\ge 3$, 
it follows from \eqref{eq:B} that $n_4=0$ and $n_3\in \{4,6\}$. Thus $|V(G)|\in \{6,8\}$.

\noindent
{\bf Case 2-1.} Suppose that $n_3=6$. 
Then $G$ is a graph with degree sequence $(3,3,3,3,3,3,2,2)$, and $H^*$ is a graph with degree sequence $(3,3,3,3,3,3)$.
If $H^*$ has a loop, then $G$ has a 3-cycle $[u_1u_2u_3]$ such that $u_1, u_2$ are 2-vertices and $u_3$ is a 3-vertex in $G$.
A subgraph of $G$ induced by $V(G) \setminus \{u_1,u_2,u_3\}$ has average degree $\frac{14}{5}$.
Since $\mad(G) < \frac{14}{5}$, $H^*$ has no loop.
Since $H^*$ has a $K_4$-minor, $H^*$ is one of graphs in \Cref{fig:33333322}: If $H^*$ has a multiple edge, then $H^*$ is  (1), and 
if $H^*$ is simple, then $H^*$ is a $2$-connected cubic graph and so $H^*$ is (2) or (3).

\begin{figure}[h!]
    \centering
    \includegraphics{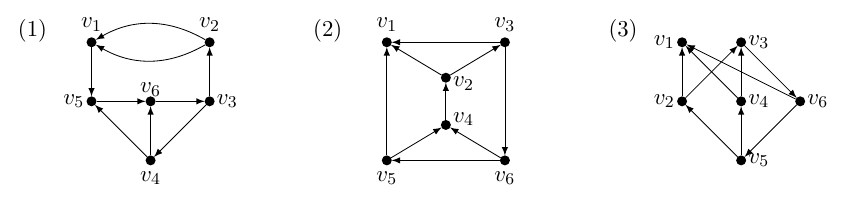}
    \caption{Orientations $D^*$ of $H^*$ when $G$ has  degree sequence $(3,3,3,3,3,3,2,2)$}
    \label{fig:33333322}
\end{figure}

By symmetry, we may assume that 
for (1), $z\in\{v_1,v_5,v_6\}$, and for (2) or (3),  $z=v_1$.    
Let $D$ be an orientation of $G$ obtained from the orientation $D^*$ of $H^*$  in \Cref{fig:33333322} by the recovering process.
In any case, $D^*-v_3v_6$ is acyclic and so any nonempty Eulerian sub-digraph of $D^*$ contains the arc $v_3v_6$. Then $D^*$ has exactly five Eulerian sub-digraphs for (1), and has exactly three Eulerian sub-digraphs for each of (2) and (3). Thus $D$ also has an odd number of 
 Eulerian sub-digraphs and so $\diff(D)\neq 0$.

\noindent
{\bf Case 2-2.} Suppose that $n_3=4$.
Then $G$ is a graph with degree sequence $(3,3,3,3,2,2)$, and $H^*$ must be $K_4$. Let $V(H^*)=\{v_1,v_2,v_3,v_4\}$. 
Since one $3$-vertex in $G$ must be $z$, we assume that $v_1=z$. Note that by the degree constraint, 
subdividing edges of $H^*$ twice makes  $G$, and so $H^*$ has at most two edges that are not edges of $G$.

\begin{figure}[h!]
    \centering
    \includegraphics{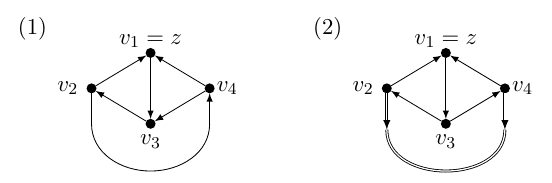}
    \caption{Orientations of $H^*$ when $G$ has degree sequence $(3,3,3,3,2,2)$}
    \label{fig:333322}
\end{figure}

Suppose that $v_2v_3,v_3v_4,v_2v_4$ are edges of $G$. Since $x,y,z$ is non-blocked, two edges of $H^*$ incident with $v_1$ are not edges of $G$.
We may assume that $v_1v_3$ and $v_1v_4$ are not edges of $G$. Let $D$ be the orientation of $G$ obtained from the orientation $D^*$ depicted in 
(1) of \Cref{fig:333322} by the recovering process. 
Then $D$ has exactly one odd Eulerian sub-digraph $v_2v_4v_3v_2$, and three even Eulerian sub-digraphs.  Thus $\diff(D)\neq 0$.

Suppose that one of $v_2v_3,v_3v_4,v_2v_4$ is not an edge of $G$, say $v_2v_4$ is not an edge. 
Let $D_0$ be the orientation depicted in (2) of \Cref{fig:333322}, where the double line shows a schematic representation of direction to define an orientation $D$ of $G$ using $D_0$. 
The path $v_2u_1\ldots u_tv_4$ of $G$ corresponding to $v_2v_4$ of $H^*$ is oriented so that $(v_2,u_1)$, $(v_4,u_t)$ are arcs of $D$ and $u_1\ldots u_t$ is a directed path. For the other edge of $H^*$ not in $G$, we naturally extend the orientation $D_0$ so that an arc of $D_0$ is a directed path in $D$. The resulting orientation $D$ of $G$ has three Eulerian sub-digraphs. Thus $\diff(D)\neq 0$.  
This completes the proof.

\section*{Acknowledgements} 

Eun-Kyung Cho was supported by the National Research Foundation of Korea (NRF) grant funded by the Ministry of Education (No. RS-2023-00244543 and No. RS-2023-00211670).
Ilkyoo Choi was supported in part by the Hankuk University of Foreign Studies Research Fund and by the Institute for Basic Science (IBS-R029-C1).
Boram Park was supported by the National Research Foundation of Korea(NRF) grant funded by the Korea government(MSIT) (No. RS-2025-00523206), and supported by the New Faculty Startup Fund from Seoul National University.
Xuding Zhu was supported by the National Natural Science Foundation of China, Grant numbers: NSFC 12371359, U20A2068.

\end{document}